\documentclass[11pt]{amsart}
\usepackage{a4,latexsym,amssymb,amsmath,amsfonts}
\newtheorem{theo}{Theorem}
\newtheorem{lem}[theo]{Lemma}
\newtheorem{prop}[theo]{Proposition}
\newtheorem{cor}[theo]{Corollary}

\newtheorem{assum}[theo]{Assumption}
\newtheorem{rem}[theo]{Remark}
\newtheorem{exa}[theo]{Example}

\numberwithin{equation}{section}
\numberwithin{theo}{section}
\def\linie{\vrule height 14pt depth 5pt}
\def\back{\noalign{\vskip-3pt}}
\title{Matrix methods for wave equations}
\author{Delio Mugnolo}
\address{Abteilung Angewandte Analysis der Universit\"at, Helmholtzstra\ss{e} 18, D-89081 Ulm, Germany and Dipartimento di Matematica dell'Universit\`a degli Studi, Via Orabona 4, I-70125 Bari, Italy}
\email{delio.mugnolo@uni-ulm.de}

\begin{document}
\maketitle

\begin{abstract}
In analogy to a characterisation of operator matrices generating $C_0$-semigroups due to R. Nagel (\cite{[Na89]}), we give conditions on its entries in order that a $2\times 2$ operator matrix generates a cosine operator function. We apply this to systems of wave equations, to second order initial-boundary value problems, and to overdamped wave equations.
\end{abstract}

\maketitle

\section{Introduction}\label{intro}

In~\cite{[Na89]}, R. Nagel started a systematic matrix theory for unbounded operators on Banach spaces. In particular, he described under which assumptions on the entries $A$, $H$, and $D$ an operator matrix \begin{equation*}
\begin{pmatrix}
A & H\\
0 & D
\end{pmatrix}
\qquad\hbox{with diagonal domain}\qquad
D(A)\times D(D)
\end{equation*}
generates a $C_0$-semigroup on a suitable product space. However, the theory presented in~\cite{[Na89]} only accounts for first order problems. In other words, the generation of cosine operator functions is not an issue there.

\bigskip
After briefly recalling in Section~\ref{cof} some known results on cosine operator functions, we state in Section~\ref{main} our main results. In analogy to the theory developed in~\cite[\S~3]{[Na89]}, we characterize when an operator matrix with diagonal domain generates a cosine operator function. In the remainder of our paper, we systematically exploit this abstract result to tackle concrete wave equations of different kinds.

\bigskip
The easiest application is to systems of wave equations, possibly on different underlying spaces. In Section~\ref{systems} we show that the well-posedness of the initial value problem associated with a system of $n$ uncoupled oscillators is not affected by the introduction of coupling terms, provided that the operators modelling such terms are not too unbounded. As a nontrivial application, we consider in Example~\ref{strom} an operator arising in fluid dynamics and already discussed by G. Str\"ohmer in~\cite{[St87]} and also, in a slightly modified setting, by Nagel in~\cite[\S~4]{[Na89]}. We show that in a $L^2$-setting Str\"ohmer's and Nagel's results can be essentially improved.

\bigskip
In Section~\ref{aibvp} we consider second order abstract initial-boundary value problems equipped with dynamic boundary conditions. Such a topic 
has aroused vast interest in recent years: we refer, e.g., to~\cite{[FGGR01]}, 
\cite{[BE04]},~\cite{[CENP05]}, and~\cite{[Mu05]}. 
The results of Section~\ref{main} allow to discuss the well-posedness of a class of wave equations with dynamical boundary conditions larger than that considered in~\cite[\S~3]{[Mu05]}. This in turn allows to improve in Example~\ref{cennorig} some results obtained by Casarino et al. in~\cite[\S~3]{[CENN03]}.

\bigskip
Finally, in Section~\ref{damped} we consider a class of second order complete abstract Cauchy problems and give a criterion for their well-posedness.
Our assumptions are in fact stronger than those imposed e.g. in~\cite[\S~2.5]{[XL98]} and~\cite[\S~VI.3b]{[EN00]}, but in this way we are able to enlarge the space on which such problems are well-posed.

\section{Basic facts on cosine operator functions}\label{cof}


Given a closed operator $A$ on a Banach space $X$, we denote by $[D(A)]$ the Banach space obtained by endowing the domain of $A$ with its graph norm. We assume the reader to be familiar with the theory of cosine operator functions as presented, e.g., in~\cite{[Fa85]} or~\cite[\S~3.14]{[ABHN01]}, and only recall the following, cf.~\cite[Theorem~3.14.11, Theorem~3.14.17, and Theorem~3.14.18]{[ABHN01]}. (If $A$ generates a cosine operator function (COF), we denote it by $(C(t,A))_{t\in\mathbb R}$, and the associated sine operator function (SOF) by $(S(t,A))_{t\in\mathbb R}$).

\begin{lem}\label{wellp2char}
Let $A$ be a closed operator on a Banach space $X$. Then the operator $A$ generates a COF on $X$ if and only if there exists a Banach space $V$, with $[D(A)]\hookrightarrow V\hookrightarrow X$, such that the operator matrix 
\begin{equation*}
{\mathbf A}:=
\begin{pmatrix}
0 & I_V\\
A & 0
\end{pmatrix},\qquad D({\mathbf A}):=D(A)\times V,
\end{equation*}
generates a $C_0$-semigroup $(e^{t\mathbf A})_{t\geq 0}$ in $V\times X$. In this case there holds
\begin{equation}\label{groupcos}
e^{t{\mathbf A}}=
\begin{pmatrix}
C(t,A) & S(t,A)\\
AS(t,A) & C(t,A)
\end{pmatrix},\qquad t\geq 0.
\end{equation}
If such a space $V$ exists, then it is unique and is called \emph{Kisy\'nski space associated with $(C(t,A))_{t\in{\mathbb R}}$}. The (unique) product space ${\bf X}=V\times X$ is called \emph{phase space associated with} $(C(t,A))_{t\in{\mathbb R}}$ (or \emph{with $A$}).
\end{lem}

\begin{lem}\label{analy}
If $A$ generates a COF on a Banach space $X$, then it also generates an analytic semigroup of angle $\frac{\pi}{2}$ on $X$. Further, the spectrum of $A$ lies inside a parabola.
\end{lem}

The following similarity and perturbation results have been proved in~\cite[\S~3.14]{[ABHN01]} and~\cite[\S~2]{[Mu05]}.

\begin{lem}\label{similcos}
Let $V_1,V_2,X_1,X_2$ be Banach spaces with $V_1\hookrightarrow X_1$ and $V_2\hookrightarrow X_2$, and let $U$ be an isomorphism from $V_1$ onto $V_2$ and from $X_1$ onto $X_2$. Then an operator $A$ generates a COF with associated phase space $V_1\times X_1$ if and only if $UAU^{-1}$ generates a COF with associated phase space $V_2\times X_2$. In this case, there holds 
\begin{equation*}
U C(t,A) U^{-1}= C(t, UAU^{-1}),\qquad t\in{\mathbb R}.
\end{equation*}
\end{lem}

\begin{lem}\label{pert}
Let $A$ generate a COF with associated phase space $V\times X$. Then also $A+B$ generates a COF with associated phase space $V\times X$, provided $B$ is an operator that is bounded from $[D(A)]$ to $V$, or from $V$ to $X$. 
\end{lem} 

\section{Main results}\label{main}

To begin with, we state an analogue of~\cite[Proposition~3.1]{[Na89]} in the context of cosine operator functions.

\begin{theo}\label{formcosm}
Let $A$ and $D$ be generators of COFs on $X$ and $Y$, respectively, with associated phase space $V\times X$ and $W\times Y$. Consider an operator $H$ that is bounded from $[D(D)]$ to $X$. Then the operator matrix 
\begin{equation*}
{\mathcal A}:=
\begin{pmatrix}
A & H\\
0 & D
\end{pmatrix},
\qquad D({\mathcal A}):=D(A)\times D(D),
\end{equation*}
generates a COF on $X\times Y$, with associated phase space $\left(V\times W\right)\times \left(X\times Y\right)$, if and only if
\begin{equation}\label{crit}
\int_0^t C(t-s,A)HS(s,D)ds,\qquad t \geq 0,
\end{equation}
can be extended to a family of linear operators from $Y$ to $X$ which is uniformly bounded as $t\to 0^+$. In this case, there holds
\begin{equation}\label{formcoss}
C(t,{\mathcal A})=
\begin{pmatrix}
C(t,A) & \int_0^t C(t-s,A)HS(s,D)ds\\
0 & C(t,D)
\end{pmatrix},
\qquad t\in{\mathbb R},
\end{equation}
(where we consider the bounded linear extension from $Y$ to $X$ of the upper-right entry) and the associated SOF is
\begin{equation}\label{formsins}
S(t,{\mathcal A})=
\begin{pmatrix}
S(t,A) & \int_0^t S(t-s,A)HS(s,D)ds\\
0 & S(t,D)
\end{pmatrix},
\qquad t\in{\mathbb R}.
\end{equation}
Such a SOF is compact if and only if the embeddings $[D(A)]\hookrightarrow X$ and $[D(D)]\hookrightarrow Y$ are both compact.
\end{theo}

\begin{proof}
The operator matrix $\mathcal A$ generates a COF with associated phase space $\left(V\times W\right)\times \left(X\times Y\right)$ if and only if the reduction matrix
$${\mathbb A}:=\begin{pmatrix}
0 & I_{V\times W}\\
{\mathcal A} & 0
\end{pmatrix},
\qquad D({\mathbb A}):=\left(D(A)\times D(D)\right)\times \left(V\times W\right),$$
generates a $C_0$-semigroup on $\left(V\times W\right)\times \left(X\times Y\right)$. Define the operator matrix
$$\mathbb{U}:=
\begin{pmatrix}
I_V & 0 & 0 & 0\\
0 & 0 & I_X & 0\\
0 & I_W & 0 & 0\\
0 & 0 & 0 & I_Y
\end{pmatrix},$$
which is an isomorphism from $\left(V\times W\right)\times \left(X\times Y\right)$ onto $\left(V\times X\right)\times \left(W\times Y\right)$ with
$$\mathbb{U}^{-1}:=
\begin{pmatrix}
I_V & 0 & 0 & 0\\
0 & 0 & I_W & 0\\
0 & I_X & 0 & 0\\
0 & 0 & 0 & I_Y
\end{pmatrix}.$$
Then the similar operator matrix $\tilde{\mathbb A}:={\mathbb U}{\mathbb A} {\mathbb U}^{-1}$ becomes
$$\tilde{\mathbb A}=
\begin{pmatrix}
{\mathbf A} & {\mathbf H}\\
0 & {\mathbf D}
\end{pmatrix},\qquad D(\tilde{\mathbb A}):=D({\mathbf A})\times D({\mathbf D}).$$
Here $\mathbf A$ and $\mathbf D$ are the reduction operator matrices
\begin{equation*}
{\mathbf A}:=
\begin{pmatrix}
0 & I_V\\
A & 0
\end{pmatrix},\qquad D({\mathbf A}):=D(A)\times V,
\end{equation*}
and
\begin{equation*}
{\mathbf D}:=
\begin{pmatrix}
0 & I_W\\
D & 0
\end{pmatrix},\qquad
D({\mathbf D}):=D(D)\times W,
\end{equation*}
respectively, while $\mathbf H$ is given by
\begin{equation*}
\qquad {\mathbf H}:=
\begin{pmatrix}
0 & 0\\
H & 0
\end{pmatrix},\qquad
D({\mathbf H}):=D({\mathbf D}).
\end{equation*}
By Lemma~\ref{wellp2char}, the operators $\mathbf A$ and $\mathbf D$ generate $C_0$-semigroups on $V\times X$ and $W\times Y$, respectively. Moreover ${\mathbf H}\in{\mathcal L}([D({\mathbf D})],V\times X)$, and a direct computation shows that 
$$e^{(t-s){\mathbf A}}{\mathbf H}e^{s{\mathbf D}}=
\begin{pmatrix}
S(t-s,A)HC(s,D) & S(t-s,A)HS(s,D)\\
C(t-s,A)HC(s,D) & C(t-s,A)HS(s,D)
\end{pmatrix},\qquad 0\leq s\leq t.$$
By virtue of \cite[Proposition~3.1]{[Na89]} we obtain that $\tilde{\mathbb A}$ generates a $C_0$-semigroup on $(W\times X)\times (W\times Y)$ if and only if the family of operators
$$\int_0^t e^{(t-s){\mathbf A}}{\mathbf H}e^{s{\mathbf D}}ds,\qquad t \geq 0,$$
from $W\times Y$ to $V\times X$ is uniformly bounded as $t\to 0^+$.
Hence, if $\tilde{\mathbb A}$ generates a $C_0$-semigroup, or equivalently if $\mathcal A$ generates a COF, then $\int_0^t C(t-s,A)HS(s,D)ds$ is uniformly bounded as $t\to 0^+$. 

Again by~\cite[Proposition~3.1]{[Na89]}
$$e^{t\tilde{\mathbb A}}=
\begin{pmatrix}
e^{t{\mathbf A}} & \int_0^t e^{(t-s){\mathbf A}}{\mathbf H}e^{s{\mathbf D}}ds\\
0 & e^{t{\mathbf D}}
\end{pmatrix},\qquad t\geq 0.$$
By similarity, $e^{t{\mathbb A}}=e^{t{\mathbb U}^{-1}\tilde{\mathbb A}{\mathbb U}}={\mathbb U}^{-1}e^{t\tilde{\mathbb A}}{\mathbb U}$, $t\geq 0$. Thus, a direct computation shows that the semigroup generated by $\mathbb A$ on the space $\left(V\times W\right)\times \left(X\times Y\right)$ is given by
\begin{equation*}
\begin{pmatrix}
\lineskip=0pt
C(t,A) & \int_0^t S(t-s,A)HC(s,D)ds &\linie & S(t,A) & \int_0^t S(t-s,A)HS(s,D)ds\\ \back
0 & C(t,D) &\linie & 0 & S(t,D)\\
\noalign{\hrule}
AS(t,A) & \int_0^t C(t-s,A)HC(s,D)ds &\linie & C(t,A) & \int_0^t C(t-s,A)HS(s,D)ds\\ \back
0 & DS(t,D) &\linie & 0 & C(t,D)\\
\end{pmatrix}
\end{equation*}
for $t\geq 0$. Since by assumption $\mathbb A$ generates a $C_0$-semigroup on the space $\left(V\times W\right)\times \left(X\times Y\right)$, comparing the above formula with~\eqref{groupcos} yields~\eqref{formcoss} and~\eqref{formsins}. 

One can also check directly that the lower-right block-entry defines a COF on $X\times Y$. Further, integrating by parts one sees that the upper-right and lower-right block-entries can be obtained by integrating the upper-left and lower-left block-entries, respectively, and moreover that the diagonal blocks coincide. Hence, by definition of SOF, all the blocks are strongly continuous families as soon as the lower-right is strongly continuous. Consequently, if the family $\int_0^t C(t-s,A) H S(s,D)ds$ is uniformly bounded as $t\to 0^+$, then the family $\int_0^t e^{(t-s){\mathbf A}}{\mathbf H}e^{s{\mathbf D}}ds$ is uniformly bounded as $t\to 0^+$.

Finally, it is known that a SOF is compact if and only if its generator has compact resolvent, cf.~\cite[Propositio~2.3]{[TW77]}, and the claim follows.\qed
\end{proof}

Observe that the operators defined in~\eqref{crit} are in general only bounded from $W$ to $X$. In Theorem~\ref{formcosm} we however required that they are bounded \emph{from the larger space} $Y$ to $X$. Such an extension can usually be performed whenever the integrated operator provides some kind of regularizing effect. However, usually COFs do not enjoys any regularity property (see~\cite[Propositio~4.1]{[TW77]}). Therefore, in most cases the following analogue of~\cite[Corollary~3.2]{[Na89]} can be applied more easily.

\begin{prop}\label{na89rev}
Let $A$ and $D$ be closed operators on the Banach spaces $X$ and $Y$, respectively. Consider further Banach spaces $V,W$ such that $[D(A)]\hookrightarrow V\hookrightarrow X$ and $[D(D)]\hookrightarrow W\hookrightarrow Y$. Assume moreover

- the operator $H$ to be bounded from $[D(D)]$ to $V$, or from $W$ to $X$, and

- the operator $K$ to be bounded from $[D(A)]$ to $W$, or from $V$ to $Y$. 

$ $ \!\!\!\!\!\!\!\!\!\!\! 
Then the operator matrix 
\begin{equation*}
\mathcal{A}:=
\begin{pmatrix}
A & H\\
K & D
\end{pmatrix},\qquad D({\mathcal A}):=D(A)\times D(D),
\end{equation*}
generates a COF with associated phase space $\left(V\times W\right)\times\left(X\times Y\right)$ if and only if $A$ and $D$ generate COFs with associated phase space $V\times X$ and $W\times Y$, respectively. In this case, $(S(t,\mathcal{A}))_{t\in\mathbb R}$ is compact if and only if the embeddings $[D(A)]\hookrightarrow X$ and $[D(D)]\hookrightarrow Y$ are both compact.
\end{prop}

\begin{proof}
The diagonal matrix 
$${\mathcal A}_0:=
\begin{pmatrix}
A & 0\\
0 & D
\end{pmatrix},
\qquad D({\mathcal A}_0):=D(\mathcal{A}),$$ 
generates a COF with associated phase space $\left(V\times W\right)\times\left(X\times Y\right)$
if and only if $A$ and $D$ generate COFs with associated phase space $V\times X$ and $W\times Y$, respectively. Consider now perturbations of ${\mathcal A}_0$ given by the operator matrices
$${\mathcal H}:=
\begin{pmatrix}
0 & H\\
0 & 0
\end{pmatrix}
\qquad\hbox{and}\qquad
{\mathcal K}:=
\begin{pmatrix}
0 & 0\\
K & 0
\end{pmatrix}.$$ 
Observe that both $\mathcal H$ and $\mathcal K$ are, by assumption, either bounded from $[D(A)]\times [D(D)]$ to $V\times W$, or from $V\times W$ to $X\times Y$. By Lemma~\ref{pert} also their sum ${\mathcal A}_0+{\mathcal H}+{\mathcal K}={\mathcal A}$ generates a COF with associated phase space $\left(V\times W\right)\times\left(X\times Y\right)$.\qed
\end{proof}



\begin{rem}\label{na89revv}
\emph{In the special case of $A\in{\mathcal L}(X)$, Proposition~\ref{na89rev} reads as follows: Let $D$ be a closed operator on $Y$, and consider a further Banach space $W$ such that $[D(D)]\hookrightarrow W\hookrightarrow Y$. Assume moreover that $H\in{\mathcal L}([D(D)],X)$ and $K\in{\mathcal L}(X,Y)$. Then
\begin{equation*}
\mathcal{A}:=
\begin{pmatrix}
A & H\\
K & D
\end{pmatrix},\qquad D({\mathcal A}):=D(A)\times D(D),
\end{equation*}
generates a COF with associated phase space $\left(X\times W\right)\times\left(X\times Y\right)$ if and only if $D$ generates a COF with associated phase space $W\times Y$.}
\end{rem}

\section{Systems of abstract wave equations}\label{systems}

We consider in this section systems of $n$ abstract wave equations and show how they can be solved by means of the results of Section~\ref{main}.

\bigskip
In the trivial case of $n$ uncoupled oscillators modelled by
\begin{equation}\tag{US$_n$}
\left\{
\begin{array}{rcll}
\ddot{u}_1(t)&=&A_1 u_1(t), &t\in\mathbb R,\\
\ddot{u}_2(t)&=&A_2 u_2(t), &t\in\mathbb R,\\
\vdots\\
\ddot{u}_n(t)&=&A_n u_n(t),\quad &t\in\mathbb R,
\end{array}
\right.
\end{equation}
it is clear that the initial value problem associated with $(\rm{US}_n)$ is well-posed (in a natural sense) if and only if each operator $A_i$, $i=1,\ldots,n$, generates a COF (with suitable associated phase spaces $V_i\times X_i$).

If however there is an interplay among the single oscillators given by
\begin{equation}\tag{CS$_n$}
\left\{
\begin{array}{ll}
\ddot{u}_1(t)=A_1 u(t)+B^1_2 u_2(t)+B^1_3 u_3(t)\ldots +B^1_n u_n(t), &t\in\mathbb R,\\
\ddot{u}_2(t)=A_2 u(t)+C^2_1 u_1(t)+B^2_3 u_3(t)\ldots +B^2_n u_n(t), &t\in\mathbb R,\\
\vdots\\
\ddot{u}_n(t)=A_n u(t)+C^n_1 u_1(t)+C^n_2 u_2(t)\ldots +C^n_{n-1} u_{n-1}(t),\quad &t\in\mathbb R,
\end{array}
\right.
\end{equation}
assumptions on the operators $B^h_k$ and $C^k_h$, $1\leq h<k\leq n$, are needed in order to obtain well-posedness.

\begin{theo}\label{too}
Let the initial value problem associated with $(\rm{US}_n)$ be well-posed, so that $V_i\times X_i$ is the phase space associated with $A_i$, $1\leq i\leq n$. Consider operators $B^h_k$ and $C^k_h$, $1\leq h<k\leq n$, such that 

- $B^h_k\in{\mathcal L}([D(A_k)],V_h)$ or $B^h_k\in{\mathcal L}(V_k,X_h)$, and 

- $C^k_h\in{\mathcal L}([D(A_h)],V_k)$ or $C^k_h\in{\mathcal L}(V_h,X_k)$. 

\. \!\!\!\!\!\!\!\!\!\!\! Then the initial value problem associated with $(\rm{CS}_n)$ is governed by a COF with associated phase space $(\prod_{i=1}^n V_i)\times (\prod_{i=1}^n X_i)$, and in particular it is well-posed.
\end{theo}

\begin{proof}
For the sake of simplicity we only discuss the case $n=2$, since the general case can be proved by induction on $n$. Consider the system
\begin{equation}\tag{CS$_2$}
\left\{
\begin{array}{ll}
\ddot{u}_1(t)=A_1 u_1(t)+B^1_2 u_2(t),\quad &t\in\mathbb R,\\
\ddot{u}_2(t)=A_2 u_2(t)+C^2_1 u_1(t), &t\in\mathbb R,\\
\end{array}
\right.
\end{equation}
which can be written as an abstract wave equation
$$\ddot{\mathfrak{u}}(t)=\mathcal{A}\mathfrak{u}(t),\qquad t\in\mathbb R,$$
on the Banach space $X_1\times X_2$, where
$$\mathcal{A}:=
\begin{pmatrix}
A_1 & B^1_2\\
C^2_1 & A_2
\end{pmatrix},
\qquad D(\mathcal{A}):=D(A_1)\times D(A_2),$$
and 
$$\mathfrak{u}:=\begin{pmatrix}u_1\\ u_2\end{pmatrix}.$$
By assumption, $A_1$ and $A_2$ generate COFs with associated phase space $V_1\times X_1$ and $V_2\times X_2$, respectively. Due to the assumptions on $B^1_2$ and $C^2_1$ the claim follows from Proposition~\ref{na89rev}.\qed
\end{proof}

%

\begin{exa}\label{strom}
\emph{Following work of A. Matsumura and T. Nishida (\cite{[NM79]}, \cite{[NM80]}), some linearized equations from fluid dynamics in an open, bounded domain $\Omega\subset \mathbb{R}^n$ lead to consider the operator matrix
$$\mathcal{A}:=\begin{pmatrix}
A_1 & B^1_2 & B^1_3\\
C^2_1 & A_2 & 0\\
C^3_1 & 0 & 0
\end{pmatrix},$$
where the operator entries $A_1,A_2,B^1_2, B^1_3, C^2_1, C^3_1$ are defined below. Such a setting has been thoroughly discussed in~\cite{[St87]} and, in a simplified and slightly different version, in~\cite[\S~4]{[Na89]}. Both authors show, by different means, that $\mathcal A$ generates an analytic semigroup on $\big(L^p(\Omega)\big)^n\times L^p(\Omega)\times W^{1,p}(\Omega)$, $1<p<\infty$; in~\cite{[St87]} some description of the spectrum of $\mathcal A$ is also given, and it is shown that the generated semigroup has angle of analyticity $\geq \frac{\pi}{4}$.}

\emph{Our aim is to show that such an operator matrix, equipped with domain
$$D({\mathcal A}):=D(A_1)\times D(A_2)\times H^1(\Omega),$$
is in fact the generator of a COF on the Hilbert space $\big(L^2(\Omega)\big)^n\times L^2(\Omega)\times H^1(\Omega)$. Consequently, by Lemma~\ref{analy} it also generates an analytic semigroup of angle $\frac{\pi}{2}$ on the same space, and moreover its spectrum is contained inside a parabola. Here 
$$A_1:=\mu_1\Delta_n+ \mu_2{\rm grad}\!\cdot\!{\rm div},\qquad D(A_1):=\big(H^2(\Omega)\cap H^1_0(\Omega)\big)^n,$$
with $(\mu_1,\mu_2)\in{\mathbb R}^2_+\setminus\{0,0\}$. If $\partial\Omega$ is smooth enough, then integrating by parts a direct computation shows that $A_1$ is the operator associated with the sesquilinear form $a$ on the Hilbert space $\big(L^2(\Omega)\big)^n$ defined by
$$a(f,g):=\sum_{i,j=1}^n\int_\Omega \mu_1\frac{\partial f_i}{\partial x_j}\overline{\frac{\partial g_i}{\partial x_j}}+ \mu_2\frac{\partial f_i}{\partial x_i}\overline{\frac{\partial g_j}{\partial x_j}}dx$$
for all
$$f:=\begin{pmatrix}
f_1\\ \vdots\\ f_n
\end{pmatrix},
g:=\begin{pmatrix}
g_1\\ \vdots\\ g_n
\end{pmatrix}\in
D(a):= \big(H^1_0(\Omega)\big)^n.$$
One sees that $a$ is symmetric, closed, and densely defined. Moreover, $a$ is positive, since $$a(f,f):=\mu_1\sum_{i=1}^n\int_\Omega \vert {\rm div} f_i\vert^2 dx + \mu_2\int_\Omega\left\vert {\rm grad} f\right\vert^2dx.$$
It is then well-known (see, e.g.,~\cite[Theorem~1.2.1]{[Da90]}) that the operator $A_1$ associated with $a$ is self-adjoint and dissipative, hence the generator of a COF with associated phase space $V_1\times X_1:=\big(H^1_0(\Omega)\big)^n\times \big(L^2(\Omega)\big)^n$.}

\emph{Further, for $\mu_3>0$, we define $A_2:=\mu_3\Delta$ on $\Omega$ equipped with either (in~\cite{[St87]}) Robin, or (in~\cite[\S~4]{[Na89]}) Dirichlet boundary conditions
. In both cases $A_2$ generates a COF, and it is well-known (see~\cite[Chapter~4]{[Fa85]}) that the associated phase space $V_2\times X_2$ is $H^1(\Omega)\times L^2(\Omega)$ or $H^1_0(\Omega)\times L^2(\Omega)$, respectively. Also any bounded operator on $H^1(\Omega)$ generates a COF with associated phase space $H^1(\Omega)\times H^1(\Omega)$.}

\emph{Define finally 
$$B^1_2:= p_1{\rm grad},\qquad B^1_3:=p_2{\rm grad},$$
and 
$$C^2_1:=p_3{\rm div},\qquad C^3_1:=p_4{\rm div},$$
where $p_1,p_2,p_3,p_4$ are constants.
Since the operator ${\rm grad}$ is bounded from $H^1(\Omega)$ to $\big(L^2(\Omega)\big)^n$, it follows that $B^1_2\in{\mathcal L}(V_2,X_1)$ and also $B^1_3\in{\mathcal L}(V_3,X_1)$. Similarly, the operator ${\rm div}$ is bounded from $\big(H^1(\Omega)\big)^n$ to $L^2(\Omega)$ as well as from $\big(H^2(\Omega)\big)^n$ to $H^1(\Omega)$, and accordingly $C^2_1\in{\mathcal L}(V_1,X_2)$ and also $C^3_1\in{\mathcal L}([D(A_1)],V_3)$.}

\emph{We conclude by Theorem~\ref{too} that the whole operator matrix $\mathcal A$ generates a COF The associated phase space is 
$$\Big(\big(H^1_0(\Omega)\big)^n\times H^1(\Omega)\times H^1(\Omega)\Big)\times \Big(\big(L^2(\Omega)\big)^n\times L^2(\Omega)\times H^1(\Omega)\Big)$$
if $A_2$ is equipped with Robin boundary conditions, or rather $$\Big(\big(H^1_0(\Omega)\big)^n\times H^1_0(\Omega)\times H^1(\Omega)\Big)\times \Big(\big(L^2(\Omega)\big)^n\times L^2(\Omega)\times H^1(\Omega)\Big)$$
if $A_2$ is equipped with Dirichlet boundary conditions.}
\end{exa}

\section{Abstract initial--boundary value problems}\label{aibvp}

We impose the following assumptions throughout this section and refer to~\cite{[CENN03]} and~\cite{[Mu05]} for motivation.

\begin{assum}\label{assdyn} $ $
\begin{enumerate}
\item $X$ and $Y$ are Banach spaces such that $Y\hookrightarrow X$.
\item $\partial X$ and $\partial Y$ are Banach spaces such that $\partial Y\hookrightarrow \partial X$.
\item ${A}:D({A})\to X$ is linear with $D(A)\subset Y$.
\item $L:D(A)\to {\partial X}$ is linear and surjective.
\item ${A}_0:=A_{\vert\ker(L)}$ is densely defined and has nonempty resolvent set.
\item $\begin{pmatrix}A\\L\end{pmatrix}:D(A)\to X\times \partial X$ is closed\footnote{Observe that, under the Assumption~\ref{assdyn}.(6), we obtain a Banach space by endowing $D(A)$ with the graph norm of ${A\choose L}$, i.e.,
$$\Vert u\Vert_{A\choose L}:=\Vert u\Vert_X +\Vert Au\Vert_X + \Vert Lu\Vert_{\partial X}.$$
We denote this Banach space by $[D(A)_L]$.} as an operator from $X$ to $X\times \partial X$.
\item $B:[D(A)_L]\to\partial X$ is linear and bounded; further, $B$ is bounded either from $[D(A_0)]$ to $\partial Y$, or from $Y$ to $\partial X$. 
\item $\tilde{B}:D(\tilde{B})\subset {\partial X}\to{\partial X}$ is linear and closed, with $D(\tilde B)\subset \partial Y$.
\end{enumerate}
\end{assum}

Under the Assumptions~\ref{assdyn} one can define a solution operator $D^{A,L}_\lambda$ of the abstract (eigenvalue) Dirichlet problem
\begin{equation}\tag{ADP}
\left\{
\begin{array}{rcl}
Au&=&\lambda u,\\
Lu&=&w,
\end{array}
\right.
\end{equation}
for all $\lambda\in\rho(A_0)$. More precisely, the following holds, cf.~\cite[Lemma~2.3]{[CENN03]} and~\cite[Lemma~3.2]{[Mu05]}.

\begin{lem}\label{dirichdef}
The problem $({\rm ADP})$ admits a unique solution $u:=D^{A,L}_\lambda w$ for all $w\in \partial X$ and $\lambda\in\rho(A_0)$. Moreover, the solution operator $D_\lambda^{A,L}$ is bounded from $\partial X$ to $Z$ for every Banach space $Z$ satisfying $D(A^\infty)\subset Z\hookrightarrow X$. In particular, $D_\lambda^{A,L}\in{\mathcal L}(\partial X,[D(A)_L])$ as well as $D_\lambda^{A,L}\in{\mathcal L}(\partial X,Y)$.
\end{lem}

Observe that, by Lemma~\ref{dirichdef}, $BD^{A,L}_\lambda\in{\mathcal L}(\partial X)$ for all $\lambda\in\rho(A_0)$.

\bigskip
We want to discuss well-posedness for a second order abstract initial-boundary value problem of the form
\begin{equation}\tag{AIBPV$^2$}
\left\{
\begin{array}{rcll}
 \ddot{u}(t)&=& {A}u(t), &t\in{\mathbb R},\\
 \ddot{w}(t)&=& Bu(t)+ \tilde{B} w(t), &t\in {\mathbb R},\\
 w(t)&=&Lu(t), &t\in {\mathbb R},\\
 u(0)&=&f, \qquad\dot{u}(0)=g,&\\
 w(0)&=&h, \qquad\dot{w}(0)=j.&
\end{array}
\right.
\end{equation}
Observe that the equations on the first and the fourth line take place on the Banach space $X$, while the remainders on the Banach space $\partial X$. We begin by re-writing $({\rm AIBVP}^2)$ as a more standard second order abstract Cauchy problem
\begin{equation}\tag{$\mathcal{ACP}^2$}
\left\{
\begin{array}{rcll}
 \ddot{\mathfrak u}(t)&=& \tilde{\mathcal A}\mathfrak{u}(t), &t\in{\mathbb R},\\
 {\mathfrak u}(0)&=&\mathfrak{f},\qquad \dot{\mathfrak u}(0)=\mathfrak{g},&
\end{array}
\right.
\end{equation}
on the product space $\mathcal{X}:=X\times {\partial X}$, where
\begin{equation}\label{acorsivo}
\tilde{\mathcal A}:=\begin{pmatrix}
A & 0\\
B & \tilde{B}
\end{pmatrix},\qquad D(\tilde{\mathcal A}):=\left\{
\begin{pmatrix}
u\\ w
\end{pmatrix}\in D(A)\times D(\tilde{B}) : Lu=w
\right\},
\end{equation}
is an operator matrix with \emph{coupled} domain on $\mathcal X$.

Here the new variable $\mathfrak{u}(\cdot)$ and the inital data $\mathfrak{f}, \mathfrak{g}$ are to be understood as
$$\mathfrak{u}(t):=\begin{pmatrix}
u(t)\\ Lu(t)
\end{pmatrix}\quad\quad\hbox{for}\quad t\in{\mathbb R},\qquad\quad 
\mathfrak{f}:=\begin{pmatrix}
f\\ h
\end{pmatrix},\quad 
\mathfrak{g}:=\begin{pmatrix}
g\\ j
\end{pmatrix}.$$
Taking the components of $(\mathcal{ACP}^2)$ in the factor spaces of $\mathcal X$ yields the first two equations in $(\rm{AIBVP}^2)$, while the coupling relation $Lu(t)=w(t)$, $t\in\mathbb R$, is incorporated in the domain of the operator matrix $\tilde{\mathcal A}$. We can thus equivalently investigate $(\mathcal{ACP}^2)$ instead of $(\rm{AIBVP}^2)$. In particular, we are interested in characterizing whether $\tilde{\mathcal A}$ generates a COF in terms of analogue properties of $A_0$ and $\tilde B$.

\bigskip
Taking into account Lemma~\ref{dirichdef} and~\cite[Lemma~3.10]{[Mu05b]} (see also~\cite[\S~2]{[En99]}), a direct matrix computation yields the following.

\begin{lem}\label{simil}
Let $\lambda\in\rho(A_0)$. Then $\tilde{\mathcal A}-\lambda$ is similar to the operator matrix
\begin{equation}\label{simileq}
{\mathcal A}_\lambda:=
\begin{pmatrix}
A_0-D_\lambda^{A,L}B-\lambda & -D_\lambda^{A,L}(\tilde{B}+BD^{A,L}_\lambda-\lambda)\\
B & \tilde{B}+BD^{A,L}_\lambda-\lambda
\end{pmatrix},
\end{equation}
with diagonal domain $D({\mathcal A}_\lambda):=D(A_0)\times D(\tilde{B})$. The similarity transformation is given by the operator 
$${\mathcal M}_\lambda:=\begin{pmatrix}
I_X & -D_\lambda^{A,L}\\
0 & I_{\partial X}
\end{pmatrix},$$ 
which is an isomorphism on $\mathcal Y:=Y\times \partial Y$ as well as on $\mathcal X$.
\end{lem}


\begin{theo}\label{main1} 
Under the Assumptions~\ref{assdyn}, the operator matrix $\tilde{\mathcal A}$ defined in~\eqref{acorsivo} generates a COF with associated phase space $\mathcal{Y}\times\mathcal{X}$ if and only if ${A}_0$ and $\tilde B$ generate COFs with associated phase space $Y\times X$ and $\partial Y\times\partial X$, respectively. In this case, $(S(t,\mathcal{A}))_{t\in\mathbb R}$ is compact if and only if the embeddings $[D(A_0)]\hookrightarrow X$ and $[D(\tilde{B})]\hookrightarrow \partial X$ are both compact.
\end{theo}

\begin{proof}
Take $\lambda\in\rho(A_0)$. By Lemma~\ref{simil} the operator matrix $\tilde{\mathcal A}-\lambda$ is similar to ${\mathcal A}_\lambda$ defined in~\eqref{simileq}, and the similarity transformation is performed by ${\mathcal M}_\lambda$, which is an isomorphism on $\mathcal X$ as well as on the candidate Kisy\'nski space $\mathcal Y$. It follows by Lemma~\ref{similcos} that $\tilde{\mathcal A}-\lambda$, and hence $\tilde{\mathcal A}$ generates a COF with associated phase space ${\mathcal Y}\times {\mathcal X}$ if and only if the similar operator ${\mathcal A}_\lambda$ generates a COF with same associated phase space. We decompose
\begin{equation*}
{\mathcal A}_\lambda=
\begin{pmatrix}
A_0 & -D^{A,L}_\lambda \tilde{B}\\
0 & \tilde{B}\\
\end{pmatrix}+
\begin{pmatrix}
-D^{A,L}_\lambda B & 0\\
B & 0\\
\end{pmatrix}+\begin{pmatrix}
-\lambda & D^{A,L}_\lambda(\lambda-BD^{A,L}_\lambda)\\
0 & BD^{A,L}_\lambda-\lambda
\end{pmatrix}.
\end{equation*}
Since the operator matrix ${\mathcal A}_\lambda$ has diagonal domain $D({\mathcal A}_\lambda)=D(A_0)\times D(\tilde{B})$, we are now in the position to apply the results of Section~\ref{main}. One sees that the second operator on the right hand side is bounded from $[D({\mathcal A}_\lambda)]$ to $\mathcal Y$ or from $\mathcal Y$ to $\mathcal X$, while the third one is bounded on $\mathcal X$. Thus, by Lemma~\ref{pert} we conclude that $\mathcal A$ generates a COF with associated phase space ${\mathcal Y}\times {\mathcal X}$ if and only if 
$$\begin{pmatrix}
A_0 & -D_\lambda^{A,L}\tilde{B}\\
0 & \tilde{B}
\end{pmatrix}\qquad\hbox{with domain}\qquad D(A_0)\times D(\tilde{B})$$
generates a COF with phase space ${\mathcal Y}\times {\mathcal X}$. Since $D^{A,L}_\lambda \tilde{B}\in{\mathcal L}([D(\tilde{B})],Y)$, the claim follows by Proposition~\ref{na89rev}.\qed
\end{proof}

By Lemma~\ref{analy} we hence obtain the following.

\begin{cor}\label{maincor} 
Under the Assumptions~\ref{assdyn}, let ${A}_0$ and $\tilde B$ generate COFs with associated phase space $Y\times X$ and $\partial Y\times\partial X$, respectively. Then the operator matrix $\tilde{\mathcal A}$ defined in~\eqref{acorsivo} generates an analytic semigroup of angle $\frac{\pi}{2}$ in $X\times\partial X$. Further, such an analytic semigroup is compact if and only if the embeddings $[D(A_0)]\hookrightarrow X$ and $[D(\tilde{B})]\hookrightarrow \partial X$ are both compact.
\end{cor}

We can now revisit a problem considered in~\cite{[CENN03]} and improve the result obtained therein.

\begin{exa}\label{cennorig}
\emph{Let $\Omega$ be a bounded open domain of ${\mathbb R}^n$ with boundary $\partial\Omega$ smooth enough, and consider the second order initial-boundary value problem}
\begin{equation}\label{cenn2}
\left\{
\begin{array}{rcll}
 \ddot{u}(t,x)&=&\Delta u(t,x), &t\in\mathbb{R},\; x\in\Omega,\\
 \ddot{w}(t,z)&=& Bu(t,z)+\Delta w(t,z), &t\in\mathbb{R},\;z\in\partial\Omega,\\
 w(t,z)&=&\frac{\partial u}{\partial \nu}(t,z), &t\in\mathbb{R},\; z\in\partial \Omega,\\
 u(0,x)&=&f(x),\qquad \dot{u}(0,x)=g(x), &x\in\Omega,\\
 w(0,z)&=&h(z), \qquad \dot{w}(0,z)=j(z), &z\in\partial\Omega.
\end{array}
\right.
\end{equation}
\emph{Set
$$X:=L^2(\Omega),\quad Y:=H^1(\Omega),\quad \partial X:=L^2(\partial \Omega),\quad\hbox{and}\quad \partial Y:=H^1(\partial\Omega).$$
Define the operators}
$$A:=\Delta,\qquad D(A):=\left\{u\in H^\frac{3}{2}(\Omega): \Delta u\in L^2(\Omega)\right\},$$
$$L:=\frac{\partial}{\partial \nu},\qquad D(L):= D(A),$$
$$\tilde{B}:=\Delta,\qquad D(\tilde{B}):=H^2(\partial\Omega),$$
\emph{i.e., $\tilde B$ is the Laplace--Beltrami operator on $\partial \Omega$.}

\emph{It has been shown in~\cite[\S~3]{[CENN03]} that $A$, $L$, and $\tilde B$ satisfy the Assumptions~\ref{assdyn}. In particular, the restriction $A_0$ of $A$ to $\ker(L)$ is the Neumann Laplacian, which generates a COF with associated phase space $H^1(\Omega)\times L^2(\Omega)$ by~\cite[Theorem~IV.5.1]{[Fa85]}. Further, the Laplace--Beltrami operator is by definition self-adjoint and dissipative on $L^2(\partial\Omega)$, and its form domain is $H^1(\partial \Omega)$. Hence $\tilde B$ generates a COF with associated phase space $H^1(\partial\Omega)\times L^2(\partial\Omega)$.}

\emph{By Theorem~\ref{main1} we conclude that the problem~\eqref{cenn2} is governed by a COF with associated phase space $(H^1(\Omega)\times H^1(\partial\Omega))\times (L^2(\Omega)\times L^2(\partial\Omega))$ whenever $B$ is a bounded operator from $H^1(\Omega)$ to $L^2(\partial\Omega)$. In other words, \eqref{cenn2} admits a unique classical solution if, in particular, $f\in H^2(\Omega)$, $g\in H^1(\Omega)$, $h\in H^2(\partial\Omega)$, and $j\in H^1(\partial\Omega)$. Finally, due to the boundednes of $\Omega$, and hence to the compactness of the embeddings $H^1(\Omega)\hookrightarrow L^2(\Omega)$ and $H^1(\partial\Omega)\hookrightarrow L^2(\partial\Omega)$, we can conclude that the SOF associated with the COF that governs~\eqref{cenn2} 
is compact.}

\emph{This also improves the result obtained for the first order case in~\cite[\S~3]{[CENN03]}.}

\end{exa}

\section{Damped problems}\label{damped}

We consider a complete second order abstract Cauchy problem
\begin{equation}\tag{${\rm cACP}^2$}
\left\{
\begin{array}{ll}
\ddot{u}(t)=Au(t)+C\dot{u}(t),\quad &t\geq 0,\\
u(0)=f,\quad\dot{u}(0)=g.&
\end{array}
\right.
\end{equation}
In the case of $C$ ``subordinated" to $A$  (i.e., when $-C$ is somehow related to a fractional power of $-A$) the well-posedness of $({\rm cACP}^2)$ has been discussed, among others, by Fattorini in~\cite[Chapter~VIII]{[Fa85]}, by Chen--Triggiani in~\cite{[CT90]}, and by Xiao--Liang in~\cite[Chapters~4--6]{[XL98]}. 

In the overdamped case (i.e., when $C$ is ``more unbounded" than $A$) the treatment is easier and several well-posedness results have been obtained, under the essential assumption that $C$ generates a $C_0$-semigroup, in~\cite{[Ne86]},~\cite{[XL98]}, and~\cite{[EN00]}.

\bigskip
A natural step is to introduce the reduction matrix
\begin{equation}\label{acompl}
{\mathcal A}:=
\begin{pmatrix}
0 & I_{D(C)}\\
A & C
\end{pmatrix},\qquad D({\mathcal A})=D(A)\times D(C).
\end{equation}
Its generator property has already been studied, under appropriate assumptions: we refer, e.g., to~\cite[\S~5--6]{[Ne86]} and~\cite[\S~VI.3]{[EN00]}. The prototype result is in fact the following (see~\cite[Theorem~2.5.2]{[XL98]}): Let $A\in{\mathcal L}(X)$. Then $({\rm cACP}^2)$ is well-posed if and only if $C$ generates a $C_0$-semigroup on $X$.

An analogue of this can be proved in the context of COFs taking into account the results of Section~\ref{main}.

\begin{prop}
Let $C$ be a closed operator on a Banach space $X$ and let $V$ be a Banach space such that $[D(C)]\hookrightarrow V\hookrightarrow X$ and $A\in{\mathcal L}(V,X)$. Then $\mathcal A$ (with domain $V\times D(C)$) generates a COF on $V\times X$ if and only if $C$ generates a COF with associated phase space $V\times X$.
\end{prop}

\begin{proof}
We can regard $0$ as a generates a COF with associated phase space $V\times V$. Since $A\in{\mathcal L}(V,X)$ and of course $I_{D(C)}\in{\mathcal L}([D(C)],V)$, it follows from Remark~\ref{na89revv} that $C$ generate a COF with associated phase space $V\times X$ if and only if $\mathcal A$ generates a COF with associated phase space $\left(V\times V\right)\times \left(V\times X\right)$.\qed
\end{proof}

We can now state the following result on very strongly damped wave equations. It generalizes the above mentioned~\cite[Theorem~2.5.2]{[XL98]} because we do not assume $A$ to be bounded on $X$.

\begin{theo}\label{corcompl}
Let $C$ generate a COF with associated phase space $V\times X$. If $A\in{\mathcal L}(V,X)$, then the operator matrix $\mathcal A$ (with domain $V\times D(C)$) defined in~\eqref{acompl} generates an analytic semigroup of angle $\frac{\pi}{2}$ on $V\times X$. 

In particular, $({\rm cACP}^2)$ is well-posed, and in fact it admits a unique classical solution $u$ for all initial data $f\in V$, $g\in X$. 
\end{theo}


\begin{exa}
\emph{Consider the initial value problem 
\begin{equation}\label{compl}
\left\{
\begin{array}{ll}
\ddot{u}(t,x)= \Delta u(t,x) -\Delta^2\dot{u}(t,x), &t\geq 0,\; x\in\Omega,\\
u(t,z)=\dot{u}(t,z)=\Delta \dot{u}(t,z)=0, &t\geq0,\; z\in\partial\Omega,\\
u(0,x)=f(x),\quad \dot{u}(0,x)=g(x),\quad &x\in\Omega.
\end{array}
\right.
\end{equation}
for an overdamped wave equation on an open, bounded domain $\Omega\subset\mathbb{R}^n$ with Lipschitz boundary.}

\emph{Set 
$$V:=H^2(\Omega)\cap H^1_0(\Omega)\qquad\hbox{and}\qquad X:=L^2(\Omega).$$
Define
$$A:= \Delta,\qquad D(A):= V,$$
$$C:= -\Delta^2,\qquad D(C):=\left\{u\in H^4(\Omega)\cap H^1_0(\Omega): \Delta u_{|\partial\Omega}=0\right\},$$ 
and observe that $-C$ is the square of $A$. Then~\eqref{compl} can be written in the abstract form $(\rm{cACP}^2)$. Since the operator $C$ is self-adjoint and strictly negative, it generates a COF with associated phase space $[D(-C)^\frac{1}{2}]\times X=(H^2(\Omega)\cap H^1_0(\Omega))\times L^2(\Omega)$. Moreover, the Laplacian $A$ is bounded from $H^2(\Omega)$ to $L^2(\Omega)$, hence we conclude by Theorem~\ref{corcompl} that the operator matrix $\mathcal A$ defined as in~\eqref{acompl} generates an analytic semigroup of angle $\frac{\pi}{2}$ on $\big(H^2(\Omega)\cap H^1_0(\Omega)\big)\times L^2(\Omega)$.} 

\emph{In particular, the problem~\eqref{compl} admits a unique classical solution for all initial data $g\in L^2(\Omega)$ and $f\in H^2(\Omega)\cap H^1_0(\Omega)$: while applying~\cite[Corollary~VI.3.4]{[EN00]} to the same problem yields existence and uniqueness of a classical solution only for $f\in D(C)$, i.e., for $u(0,\cdot)$ in a class of $H^4(\Omega)$-functions.}
\end{exa}

\end{document}